\documentclass[a4paper, 12pt]{article}

\usepackage[latin1]{inputenc}
\usepackage[T1]{fontenc}
\usepackage{graphicx}
\usepackage{amsmath}
\usepackage{amssymb}
\usepackage{amsthm}
\usepackage{hyperref}

\numberwithin{equation}{section}

\newtheorem{theorem}{Theorem}
\newtheorem{proposition}[theorem]{Proposition}
\newtheorem{lemma}[theorem]{Lemma}
\newtheorem{conjecture}[theorem]{Conjecture}
\newtheorem*{question}{Question}
\newtheorem{corollary}[theorem]{Corollary}

\theoremstyle{definition}
\newtheorem{example}[theorem]{Example}

\newcommand{\abs}[1]{\lvert #1 \rvert}
\newcommand{\cluster}{\mathcal{C}}

\title{A Tur\'an-type problem on distances in graphs}
\author{Mykhaylo Tyomkyn$^{{\rm a}}$, Andrew Uzzell$^{{\rm b}}$\thanks{Research supported in part by NSF grant DMS-0505550}\\
\small $^{\rm a}${\it University of Cambridge, Department of Pure Mathematics and Mathematical Statistics}\\
[-0.8ex]
\small{\it Centre for Mathematical Sciences, Wilberforce Road}\\
[-0.8ex]
\small{\it Cambridge, CB3 0WB, England}\\
\small $^{\rm b}${\it Department of Mathematics, Uppsala University}\\
[-0.8ex]
\small{\it PO Box 480, 751 06, Uppsala, Sweden}\\
}

\begin{document}

\date{\today}

\maketitle

\begin{abstract}
We suggest a new type of problem about distances in graphs and make several conjectures. As a first step towards proving them, we show that for sufficiently large values of $n$~and~$k$, a graph on $n$ vertices that has no three vertices pairwise at distance~$k$ has at most~$(n-k+1)^2 / 4$ pairs of vertices at distance~$k$. 
\end{abstract}

\section{Introduction}\label{intro}

In~\cite{bb-mt}, Bollob\'as and Tyomkyn determined the maximum number of paths of length~$k$ in a tree~$T$ on $n$ vertices. Here we suggest an extension of this problem to general graphs.

The `obvious' extension, counting paths of a given length in a graph~$G$, has been studied since 1971, see, e.g.,~\cite{ahl-kat78,alon81,alon86,bb-erd,bb-sar1,bb-sar2,byer,fur92,katz71} and the references therein. On the other hand, counting paths of length~$k$ in \emph{trees} can be interpreted as counting pairs of vertices at distance~$k$.  Therefore, a natural question to ask is the following.

\begin{question}
For a graph~$G$ on $n$ vertices, what is the maximum possible number of pairs of vertices at distance~$k$?
\end{question}

To the best of our knowledge, this question has not been considered previously. Our aim in this paper is to formulate several conjectures and to prove one of them in the first non-trivial special case.

For a graph~$G$, define the distance-$k$ graph~$G_k$ to be the graph with vertex set~$V(G)$ and $\{x,y\} \in E(G_k)$ if and only if $x$ and $y$ are at distance~$k$ in $G$, that is, the shortest path between $x$ and $y$ has length~$k$. We call such vertices $x$ and~$y$ \emph{$k$-neighbours} and the pair~$\{x, y\}$ a \emph{$k$-distance.}  We call $d_{G_k}(x)$ the \emph{$k$-degree} of~$x$.

Observe that if $H$ is an induced subgraph of~$G$, then $H_2$ is a subgraph of~$G_2$.  This need not be the case when $k \geq 3$.  It is clear that $G_k \cong H_k$ does not imply that $G \cong H$.  If $G_k \cong H_k$, then we say that $G$ is \emph{$k$-isomorphic} to $H$.

We wish to maximise the number of edges in $G_k$ over all graphs~$G$ on $n$ vertices. One attempt to construct a graph with many $k$-neighbours would be to consider what we call \emph{$t$-brooms}. For even $k \geq 4$ and for $t \geq 2$, define a $t$-broom to be a graph consisting of a central vertex~$v$ with $t$ `brooms' attached, each consisting of a path on $(k - 2)/2$ vertices with leaves attached to the ends opposite $v$. In this way, the leaves of different brooms will be at distance~$k$. For odd $k \geq 3$, to define a $t$-broom, take a copy of~$K_t$ and attach a broom to each vertex, adjusting the length of the path. (See Figure~\ref{fig:tBroomFig}.) As in Tur\'an's theorem, the number of $k$-distances in a $t$-broom will be maximised when the numbers of leaves in the brooms are as equal as possible.

\begin{figure}
\centering
\begin{tabular}{ccc}
\includegraphics[scale=0.5]{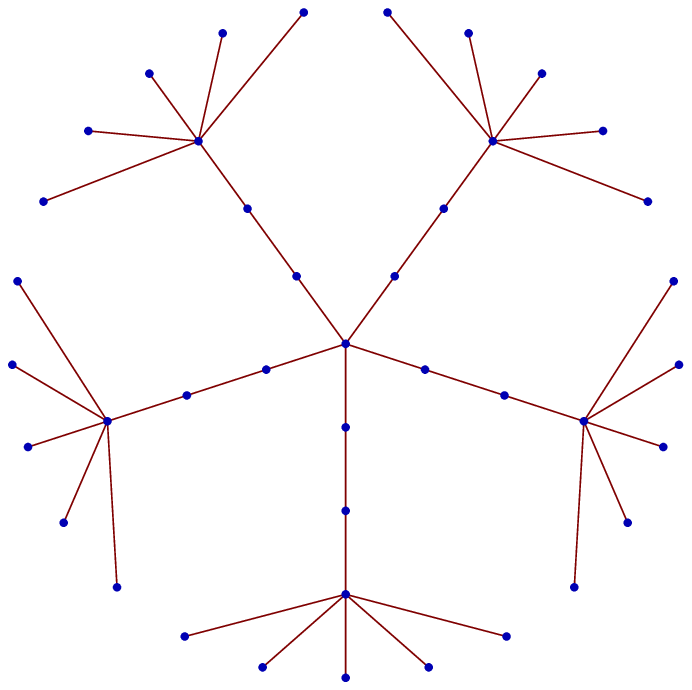} &
\hspace{1cm}
&
\includegraphics[scale=0.5]{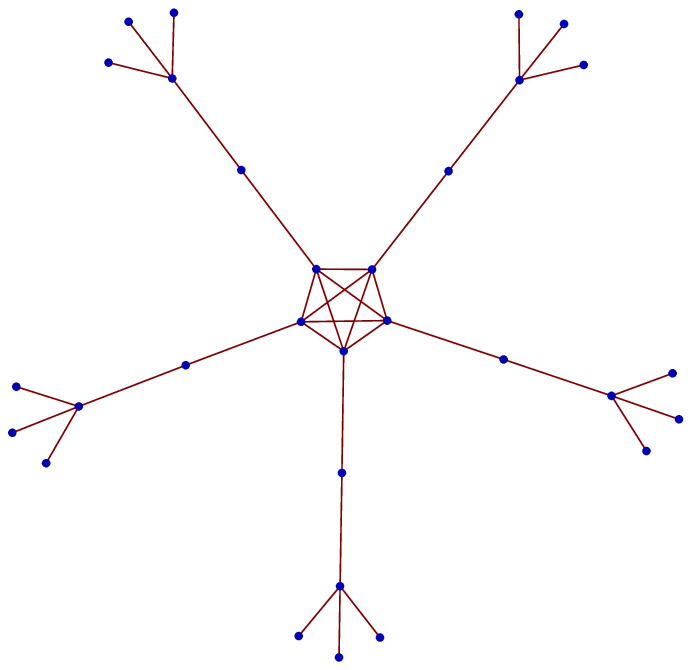}
\end{tabular}
\caption[Examples of $t$-brooms]{A 5-broom for $k = 8$ and a 5-broom for $k = 7$.}\label{fig:tBroomFig}
\end{figure}

In~\cite{bb-mt} Bollob\'as and Tyomkyn proved that if $G$ is a tree, then $e(G_k)$ is maximal when $G$ is a $t$-broom for some~$t$.

\begin{theorem}\label{thm:treemax} 
Let $n \geq k$.  If $G$ is a tree on $n$ vertices, then $e(G_k)$ is maximal when $G$ is a $t$-broom. If $k$ is odd, then $t=2$. If $k$ is even, then $t$ is within~$1$ of
\[\dfrac{1}{4} + \sqrt{\dfrac{1}{16} + \dfrac{n - 1}{k - 2}}.\]
\end{theorem}

These results prompt us to make the following conjecture.

\begin{conjecture}\label{General}
Let $k \geq 3$. There exists~$h = h(k)$ such that if $n \geq h(k)$, then $e(G_k)$ is maximised over all $G$ with $\lvert G \rvert = n$ when
$G$ is $k$-isomorphic to a $t$-broom
for some~$t$.
\end{conjecture}

For small values of~$n$ there exist better constructions. For example, if $k = 3$ and $n = 7$, the $7$-cycle has more $3$-distances than any $t$-broom.

We firmly believe Conjecture~\ref{General} to be true, but are unable to prove it.  In this paper, we approach Conjecture~\ref{General} by placing a restriction on $\omega(G_k)$, the clique number of~$G_k$, which is the maximal number of vertices at pairwise distance~$k$. We formulate the following natural analogue of Conjecture~\ref{General} under this condition. 

\begin{conjecture}\label{Kt+1Free}
Let $k\geq 3$ and $t\geq 2$. There is a function~$h_2\colon \mathbb{N}\times \mathbb{N}\rightarrow \mathbb{N}$ such that if $n \geq h_2(k,t)$, then $e(G_k)$ is maximised over all $G$ with $\lvert G \rvert = n $ and $\omega(G_k)\leq t$ when
$G$ is $k$-isomorphic to a $t$-broom
for some~$t$. 
\end{conjecture}

In this paper, we shall discuss the case $t = 2$ of Conjecture~\ref{Kt+1Free}, that is, the case when no three vertices of~$G$ are pairwise at distance~$k$. Note that in this case the parity of~$k$ matters little, as the conjectured optimal example is just a path of length~$k - 2$ with equally many leaves attached to each of its endvertices. The number of $k$-distances in such a graph is simply $\lfloor(n - k + 1)^2/4\rfloor$.

We prove Conjecture~\ref{Kt+1Free} for $t=2$ and sufficiently large~$k$. More precisely, we prove the following assertion.

\begin{theorem}\label{thm:Dbroom}
There is a constant~$k_0$ and a function~$n_0 \colon \mathbb{N}\rightarrow \mathbb{N}$ such that for all~$k\geq k_0$, all~$n\geq n_0(k)$ and all graphs~$G$ of order~$n$ with no three vertices pairwise at distance~$k$,
\[e(G_k)\leq \dfrac{\left(n - k + 1\right)^2}{4}.\]
Moreover, if equality holds, then
$G$ is $k$-isomorphic to the double broom.
\end{theorem}

We do not make any effort to determine $k_0$ and $n_0$ exactly; on the contrary, we are rather generous about them. However, we conjecture that $k=2$ and $(n,k)=(7,3)$ are the only exceptions to the optimality of the double-broom.

\begin{conjecture}\label{TriangleFree}
In the setting of Theorem~\ref{thm:Dbroom} we can take $k_0 = 3$, $n_0(3) = 8$ and $n_0(k) = k+1$ otherwise.
\end{conjecture}

For $k = 2$, we can do a little better than the bound in Theorem~\ref{thm:Dbroom}, as the following construction shows.

\begin{example}\label{ex:G2max}
Let $X$ and $Y$ be cliques on $(n + 1)/2$ vertices each, with a vertex~$z$ in common. Take vertices $x \in X$ and~$y \in Y$. Remove the edges $\{x,z\}$ and~$\{y,z\}$ and add the edge~$\{x,y\}$; call the resulting graph~$G$. Then $G_2$, the distance-$2$ graph of~$G$, is a complete bipartite graph with one edge subdivided, and thus has a total of~$(n - 1)^{2}/4 + 1$ edges.
\end{example}

We believe that for $n \geq 5$, a triangle-free $G_2$ can have no more than~$(n - 1)^{2}/4 + 1$ edges.  This is clearly true for $n = 5$, and a computer search verifies that it also holds for $6 \leq n \leq 11$.  However, we cannot prove that it is true in general.  If this is indeed so, then it shows that for $k = 2$ the quantity~$(n - k + 1)^2/4$ is within~$1$ of the maximum. The same would hold for $k=3$ and $n=7$, when the aforementioned $7$-cycle wins by~$1$ over the double broom.

In the general case, with no restriction on $\omega(G_2)$, the maximum value of~$e(G_2)$ is straightforward.

\begin{proposition}
Let $G$ be a graph on $n$ vertices.  Then
\[e(G_2) \leq \dbinom{n - 1}{2}.\]
Moreover, if equality holds, then $G$ is a star.
\end{proposition}

\begin{proof}
The result is clear for $n = 3$. We suppose that for some $n > 3$, the result holds for all graphs on at most~$n-1$ vertices. Let $G$ be a graph on $n$ vertices. We may assume that $G$ is connected. Otherwise, if $G$ has $\ell \geq 2$ components of sizes $c_1$, $\ldots\,$,~$c_{\ell}$, say, then, by hypothesis,
\[e(G_2) \leq \sum_{i=1}^{\ell} \dbinom{c_i - 1}{2} < \dbinom{n-1}{2}.\]
Observe that, setting $d = \operatorname{diam}(G)$,
\[\dbinom{n}{2} = e(G) + e(G_2) + \cdots + e(G_d).\]
Thus,
\begin{equation}\label{eq:EdgeDecomp}
e(G_2) \leq \dbinom{n}{2} - e(G).
\end{equation}
Since $G$ is connected, $e(G) \geq n - 1$, and hence $e(G_2) \leq \binom{n - 1}{2}$. Note that if equality holds, then $G$ must be a tree. Moreover, in this case, by~\eqref{eq:EdgeDecomp}, $\operatorname{diam}(G) = 2$, so $G$ must be a star.
\end{proof}

\section{Preliminaries}\label{prelim}

In this section, we shall prove a straightforward bound on $e(G_k)$.  We shall then discuss possible ways of extending this result to the upper bound in Theorem~\ref{thm:Dbroom}.  We shall also discuss a useful property of spanning trees.

Since $G_k$ is triangle-free by our assumption, Mantel's theorem implies that $e(G_k)\leq n^2/4$. In fact, we can do somewhat better by adapting a standard proof of Mantel's Theorem (see, e.g.,~\cite{AigTuran}) to $k$-distances.

\begin{lemma}\label{CS}
If $G_k$ is triangle-free, then
\[e(G_k)\leq \dfrac{n(n - k + 1)}{4}.\]
\end{lemma}

\begin{proof}
For a vertex~$x$, let $\nu(x)$ be the number of $k$-neighbours of~$x$, or equivalently, the degree of~$x$ in $G_k$. We may bound $e(G_k)$ as follows: for each pair of vertices $x$,~$y$ at distance~$k$, count the $k$-neighbours of~$x$ and of~$y$. Note that since $x$ and $y$ have no common $k$-neighbours (otherwise there would be a triangle in $G_k$), we have $\nu(x) + \nu(y)\le n$. In fact, we can claim that $\nu(x) + \nu(y)\leq n - k + 1$, since none of the $k - 1$ internal vertices on the shortest path between $x$ and $y$ are $k$-neighbours of~$x$ or of~$y$. Summing over all such pairs~$\{x,y\}$, we obtain
\begin{equation}\label{eq:CS1sthalf}
(n - k + 1) e(G_k) \geq \sum_{\{x,y\}\in E(G_k)} \bigl(\nu(x) + \nu(y)\bigr).
\end{equation}

Observe that for each $x \in V(G)$, the quantity~$\nu(x)$ appears $\nu(x)$ times on the right-hand side of~\eqref{eq:CS1sthalf}. By the Cauchy-Schwarz inequality, we have 
\begin{equation}\label{eq:CSsecondhalf}
\sum_{\{x,y\}\in E(G_k)} \bigl(\nu(x) + \nu(y)\bigr) = \sum_{x \in V(G)} \nu(x)^2 \geq \dfrac{1}{n} \Biggl(\sum_{x \in V(G)} \nu(x)\Biggr)^{\!2} = \dfrac{4}{n} \bigl(e(G_k)\bigr)^2.
\end{equation}

It follows from~\eqref{eq:CS1sthalf} and~\eqref{eq:CSsecondhalf} that
\begin{equation}\label{eq:CSbd}
e(G_k) \leq \dfrac{n(n - k + 1)}{4},
\end{equation}
as claimed.
\end{proof}

The bound that we have just proved is about halfway between the trivial $n^2/4$ and the desired $(n - k + 1)^2/4$. There are two natural ways in which one could try to improve~\eqref{eq:CSbd}. One is to try to find vertices that have no $k$-neighbours at all. We say that such a vertex is an {\em interior} vertex; otherwise, a vertex is called an {\em exterior} vertex.  The inspiration for this terminology is as follows.  If $k \leq \operatorname{diam}(G) \leq 2k - 1$, then a vertex at or near the centre of the graph has no $k$-neighbours, while a vertex that is far from the centre of the graph will have one or more $k$-neighbours.  If $G$ has at least~$r$ interior vertices, then adapting the proof of Lemma~\ref{CS} improves the bound in~\eqref{eq:CSbd} to
\begin{equation}\label{eq:CSint}
e(G_k) \leq \dfrac{(n - r)(n - k + 1)}{4}.
\end{equation}
Thus, we are done if we can find at least~$k - 1$ interior vertices (which holds in the case when $G$ is the double-broom). So, we may assume that $r < k - 1$.

The other way to improve~\eqref{eq:CSbd} would be to find many pairs of vertices~$\{u,v\}$ such that many vertices~$z$ are $k$-neighbours of neither $u$ nor~$v$. We say that a vertex~$v \in V$ is {\em $k$-unaffiliated} with a vertex~$u \in V$ if $d(u,v) \neq k$. A vertex~$v$ is {\em $k$-unaffiliated} with a set~$U \subseteq V$ if it is $k$-unaffiliated with each $u \in U$. Otherwise, we say that $v$ is {\em $k$-affiliated} with~$U$. Thus, an interior vertex is $k$-unaffiliated with~$V$. If $G$ has $r$ interior vertices and each $\{u,v\}\in E(G_k)$ has $p$ $k$-unaffiliated vertices, then~\eqref{eq:CSbd} improves to
\begin{equation}\label{eq:CSintunaff}
e(G_k) \leq \dfrac{(n - r)(n - p)}{4} \leq \dfrac{\left(n - \dfrac{r + p}{2}\right)^2}{4}. 
\end{equation}
In particular, we are done if $p \geq 2k - r - 2$, i.e., if every $\{u,v\}\in E(G_k)$ has at least~$k - r - 1$ $k$-unaffiliated vertices other than those on the shortest path between $u$ and $v$.

In order to prove Theorem~\ref{thm:Dbroom}, we shall show that for every pair of $k$-neighbours in $G$, $p \geq 2k - r - 2$.  To do so, we shall need to study paths in $G$.  In connection with this, we shall need the following easy result about the lengths of paths in a spanning tree.  If $P$ is a path, we write $\abs{P}$ to denote the number of vertices in $P$.

\begin{lemma}\label{SpTreePathLen}
Let $r \geq 2$ and let $G$ be a graph on at least~$r + 1$ vertices.  If $G$ has at most~$r$ interior vertices, then every spanning tree of~$G$ either contains no path of length at least~$r + 1$ or contains a path of length~$2k - r$.
\end{lemma}

\begin{proof} 
Let $T$ be a spanning tree of~$G$ such that $T$ contains a path of length at least~$r + 1$. Let $P$ be a longest path in $T$ and let $u$ and $v$ be its endpoints. Suppose that $\lvert P \rvert < 2k - r$. Let $x_1$, $\ldots\,$,~$x_{r + 1}$ be the $r + 1$ central vertices of~$P$, indexed consecutively in order of increasing distance from $u$.  (If $\abs{P}$ is even, then there are two choices for the $r + 1$ central vertices of~$P$; choose one arbitrarily.)  Suppose that some $x_i$ had a $k$-neighbour in $G$, called $y$. Let $Q$ be the path in $T$ from $x_i$ to~$y$. Since, for all~$x$,~$y \in G$, $d_G(x,y) \leq d_T(x,y)$, we know that $\abs{Q} \geq k + 1$. Let $z$ be the furthest vertex from $x_i$ at which $P$ and $Q$ coincide; by symmetry, we may assume that $z$ is closer to~$v$ than $x_i$ is. We shall construct a path in $T$ that is longer than~$P$, in contradiction to the assumption. Define $P_{+}$ to be the path formed by $Q$ and the portion of~$P$ between $u$ and $x_i$. Then
\[
\abs{P_{+}} = \abs{Q} + d_T(x_1, x_i) + d_T(u, x_1) \geq (k + 1) + (i - 1) + d_T(u, x_1).
\]
Because $x_1$, $\ldots\,$,~$x_{r+1}$ are the $r + 1$ central vertices of~$P$, we have
\[
d_T(u, x_1) \geq \left\lfloor \dfrac{\abs{P} - 1 - d_T(x_1, x_{r + 1})}{2} \right\rfloor \geq \dfrac{\abs{P} - r - 2}{2},
\]
from which we deduce that
\[
\abs{P_{+}} \geq k + i + \dfrac{\abs{P} - r - 2}{2} \geq \dfrac{\abs{P} + 2k - r}{2} > \abs{P}.
\]
Thus, no $x_i$ can have a $k$-neighbour in $G$, which means that all of the $x_i$ must be interior vertices.
\end{proof}

The following concepts will play key roles in the proof of Theorem~\ref{thm:Dbroom}.  Let $v$ and $w$ be two vertices at distance~$k$.  Recall that a \emph{geodesic} is a shortest path between two vertices in a graph.  Define the {\em $vw$-path} $P$ to be a shortest path between $v$ and $w$.  Order the vertices of~$G$ and conduct a breadth-first search starting from $v$ such that the resulting tree~$T_v$ contains $P$. We define the {\em $v$-path}~$P_v$ to be the longest path in $T_v$.  (If there is more than one longest path in $T_v$, then we choose one arbitrarily.)  Similarly, we define $T_w$ to be breadth-first tree with respect to $w$ containing $P$, and the {\em $w$-path}~$P_w$ to be the longest path in $T_w$.

Consider the breadth-first tree~$T_v$ and the $v$-path~$P_v$.  Let $x$ and $y$ be the endpoints of~$P_v$.  Moving along $P_v$ from $x$ to~$y$, or in fact along any path in $T_v$, the distance from $v$ will first decrease, then increase --- this is a fundamental property of breadth-first search trees, for the depth of a vertex~$w$ in such a tree equals~$d_G(v, w)$.  Thus, $P_v$ can be divided into two geodesics.  Let $z$ be a vertex of~$P_v$ at minimal distance from $v$.  Let $P_1$ denote the portion of~$P_v$ between $x$ and $z$ and $P_2$ the portion of~$P_v$ between $z$ and $y$; one of these may be empty.

By our assumption that $r < k - 1$, if $v$ and $w$ are $k$-neighbours, then $T_v$ must contain a path of length at least~$r + 1$.  Hence, by Lemma~\ref{SpTreePathLen}, $P_v$ contains at least~$2k - r$ vertices.  Note also that at most two vertices on $P_v$ can be at distance~$k$ from $v$.

\section{Proof of Theorem~\ref{thm:Dbroom}}\label{proofmain}

From now on, we shall assume that $G$ satisfies $e(G_k) \geq (n - k + 1)^2 / 4$.  As noted above, we shall also assume that $r < k - 1$.

Let us briefly discuss the proof of Theorem~\ref{thm:Dbroom}.  First, for $k$ and $n$~large enough, we shall prove a simple condition under which a pair of $k$-neighbours must have at least~$2k - r - 2$ $k$-unaffiliated vertices. We shall deduce from this that each pair of $k$-neighbours in $G$ has almost enough $k$-unaffiliated vertices to achieve the desired bound on $e(G_k)$. Second, we shall deduce our key lemma, which says that if some pair of $k$-neighbours~$\{v, w\}$ does not have enough $k$-unaffiliated vertices, then all geodesics in $G$ must have only a few vertices apart from each of $P_v$~and~$P_w$. Third, we shall show that in this case, every other $k$-neighbour of~$v$ is at distance~$o(k)$ from $w$, and vice-versa.  Moreover, for some~$\delta = o(1)$, we shall show that all vertices that are at distance at least~$\delta k$ from both $v$ and $w$ have very few $k$-neighbours, which, by Tur\'an's theorem, will contradict our assumption that $e(G_k) \geq (n - k + 1)^2 / 4$. Finally, we shall deduce that in this case, $G$ has at most as many $k$-distances as the double broom, which will imply that $e(G_k) = (n - k + 1)^2 / 4$.  Moreover, we shall show that in this case
$G$ is $k$-isomorphic to the double broom.

\begin{lemma}\label{AffPtsGeodesic}
For every $\varepsilon > 0$ there exists a constant~$K(\varepsilon)$ such that for all~$k \geq K(\varepsilon)$, if some geodesic in $G$ contains $\varepsilon k$ vertices that are $k$-affiliated with either $v \in E(G_k)$ or~$w \in E(G_k)$, then we can find $2k$ vertices that are $k$-unaffiliated with both $v$~and~$w$.
\end{lemma}

\begin{proof}
Suppose that $Q$ is a geodesic and that $\varepsilon k$ vertices of~$Q$ are $k$-affiliated with either $v$ or~$w$. By the pigeonhole principle, we can assume that $m\geq \varepsilon k/2$ of them are at distance~$k$ from $w$. Index them consecutively by $x_1$, $\ldots\,$,~$x_m$.  For $1 \leq i \leq m$, let $Q_i$ denote a shortest path in $G$ from $x_i$ to~$w$.  Let us $7$-colour the vertices of~$Q$ so that $x_i$ is coloured with colour~$j$ if $i \equiv j \bmod{7}$ and choose a colour (red, say) that belongs to at least~$m/7$ of the $x_i$.  For each red~$x_i$, move two steps along the path~$Q_i$. In this way, we obtain $m' \geq m/7$ vertices at distance~$k - 2$ from $w$ and at distance at least~$7 - 2 - 2 = 3$ from each other; let us call them $y_1^1$,~$y_2^1$, $\ldots\,$,~$y_{m'}^1$. If a vertex~$y_i^1$ is not at distance~$k$ from $v$, set $z_i^1 = y_i^1$. If $y_i^1$ is at distance~$k$ from $v$, take $z_i^1$ to be the vertex obtained by moving from $y_i^1$ one step towards $v$ (see Figure~\ref{fig:geodesic}). Since the $y_i^1$ were at pairwise distance at least~$3$, all of the $z_i^1$ will be distinct vertices at distance between $k - 1$~and~$k - 3$ from $w$ and not at distance~$k$ from $v$, i.e., they are $k$-unaffiliated with both $v$~and~$w$.

\begin{figure}
\centering
\includegraphics{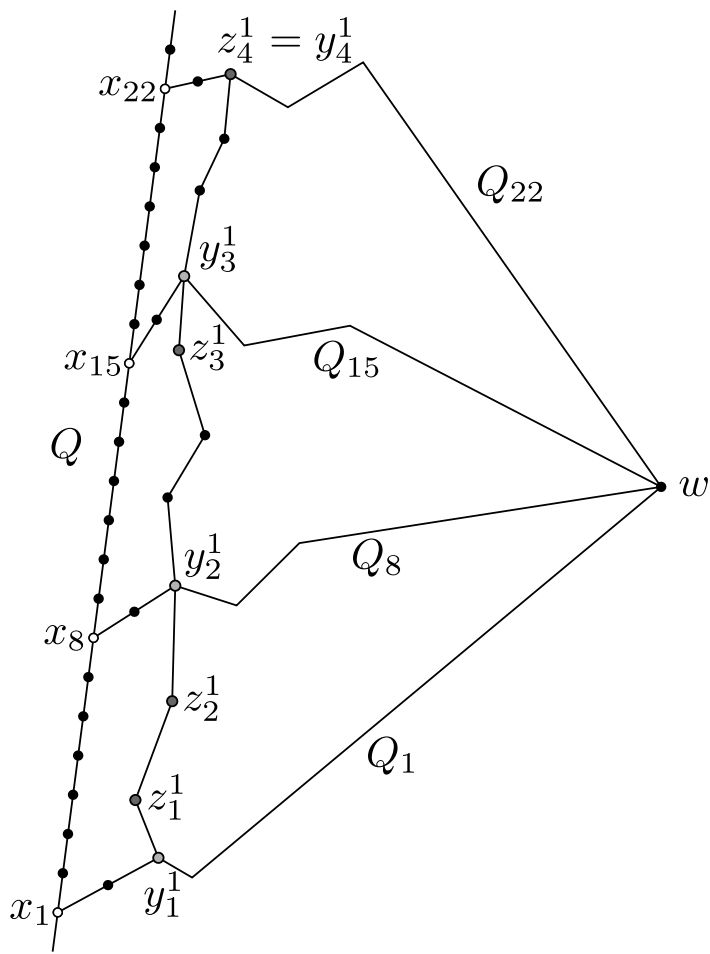}
\caption[The set of vertices that are $k$-unaffiliated with both $v$~and~$w$]{The set of vertices~$z_i^{1}$ that are $k$-unaffiliated with both $v$~and~$w$.}\label{fig:geodesic}
\end{figure}

Similarly, by 13-colouring $Q$, defining `red' to be the largest colour class, and moving 5 steps along the paths~$Q_i$, we find $m''\geq m/13$ vertices $y_1^2$,~$y_2^2$, $\ldots\,$,~$y_{m''}^2$ at distance~$k - 5$ from $w$ and at distance at least~$13 - 5 - 5 = 3$ from each other. This gives rise to~$m/13$ distinct $k$-unaffiliated vertices~$z_i^2$ at distance between $k - 4$~and~$k - 6$ from $w$ and not at distance~$k$ from $v$, i.e., the $z_i^2$ are disjoint from the previously constructed~$z_i^1$. 

Repeating this procedure for all $6\ell + 1$-colourings up to $\ell = \lfloor m/6\rfloor$, we obtain a total of at least
\[\left(\dfrac{1}{7} + \dfrac{1}{13} + \cdots + \dfrac{1}{6\left\lfloor m/6\right\rfloor + 1}\right)m \geq \dfrac{1}{12}\cdot\dfrac{\varepsilon k}{2}\log \dfrac{\varepsilon k}{2}\]
$k$-unaffiliated vertices, which is greater than~$2k$ for $k$~large enough.
\end{proof}

Lemma~\ref{AffPtsGeodesic} has the following important corollary.

\begin{corollary}\label{PvManyUnaff}
Let $\{v,w\}\in G_k$.  Let $P_v$ and $P_w$ be geodesics as defined above.  Then either $\lvert P_v \rvert - \varepsilon k \geq (2k - r) - \varepsilon k$ vertices on $P_v$ are $k$-unaffiliated with $v$~and~$w$ or we can find $2k$ vertices that are $k$-unaffiliated with $v$~and~$w$ elsewhere.  The same is true of~$P_w$. \qed
\end{corollary}

Corollary~\ref{PvManyUnaff} and equation~\eqref{eq:CSintunaff} have the following immediate consequence, which is an approximate version of the bound in Theorem~\ref{thm:Dbroom}.

\begin{corollary}\label{ApproxBound}
For every $\varepsilon > 0$ there is a constant~$K(\varepsilon)$ such that for all~$k \geq K(\varepsilon)$, if $G_k$ is triangle-free then
\[e(G_k)\leq \dfrac{\bigl(n - (1 - \varepsilon)k\bigr)^2}{4}. \tag*{\qedsymbol}
\]
\end{corollary}

Lemma~\ref{AffPtsGeodesic} and Corollary~\ref{PvManyUnaff} also give us very useful information about the structure of the graph.  The following lemma is the main tool in the remainder of the proof of Theorem~\ref{thm:Dbroom}.

\begin{lemma}\label{GeodesicDisj}
Let $\varepsilon > 0$ and let $k \geq K(\varepsilon)$, where $K(\varepsilon)$ is as in Lemma~\ref{AffPtsGeodesic}.  Suppose that $v$ and $w$ are $k$-neighbours in a graph~$G$ that have fewer than~$2k - r - 2$ $k$-unaffiliated vertices. Then any geodesic in $G$ contains fewer than~$2\varepsilon k$ vertices disjoint from~$P_v$, and similarly for $P_w$.
\end{lemma}

\begin{proof}
Suppose that $Q$ is a geodesic in $G$ with at least~$2\varepsilon k$ vertices disjoint from~$P_v$. First, if $\varepsilon k$ of these vertices are $k$-unaffiliated with $v$~and~$w$, then, by Corollary~\ref{PvManyUnaff}, $v$ and $w$ have at least~$2k \geq 2k - r - 2$ $k$-unaffiliated vertices. Second, if not, then $Q \setminus P_v$ must contain at least~$\varepsilon k$ vertices that are $k$-affiliated with $v$~and~$w$. In this case, by Lemma~\ref{AffPtsGeodesic}, we again obtain at least~$2k \geq 2k - r - 2$ $k$-unaffiliated vertices. In either case, we reach a contradiction.
\end{proof}

Let $P$ denote the $vw$-path and let $P'$ denote the $vw'$-path, as defined in Section~\ref{prelim}.  Recall that the $v$-path~$P_v$ splits into two geodesics, which we call $P_1$ and $P_2$, along each of which the distance from $v$ is strictly monotone. Similarly, the $w$-path~$P_w$ splits into two geodesics, which we call $P_3$ and $P_4$.  Note that since $P_v$ was defined on a tree~$T_v$ that contains $P$, we have that $P_v\cap P$ is an interval of~$P$, lying entirely in $P_1$ or in $P_2$, and analogously for $P_w \cap P$.  For the remainder of the proof, without loss of generality, let $P_v \cap P \subseteq P_1$ and let $P_w \cap P \subseteq P_3$.

As was already mentioned in Section~\ref{prelim}, by Lemma~\ref{CS} we are done if any pair of $k$-neighbours has $2k - r - 2$ $k$-unaffiliated vertices. So let us assume for the sake of contradiction that some vertices $v$ and~$w$ at distance~$k$ have fewer than that many $k$-unaffiliated vertices. Let $w'$ be another $k$-neighbour of~$v$. How large can the distance between $w$ and $w'$ be? The following lemma shows that this distances is either close to~$2k$ or close to~$0$.

\begin{lemma}\label{distalt}
Let $\varepsilon > 0$ and let $k \geq K(\varepsilon)$, where $K(\varepsilon)$ is as in Lemma~\ref{AffPtsGeodesic}.  Let $v$ and $w$ be $k$-neighbours with fewer than~$2k - r - 2$ $k$-unaffiliated vertices.  Let $w'$ be another $k$-neighbour of~$v$.  Then either $d(w,w') = (2 - o(1))k$ or $d(w,w') = o(k)$.
\end{lemma}

\begin{proof}
Recall that $P_v$ and $P_w$ denote the $v$-path and the $w$-path, respectively. By Corollary~\ref{PvManyUnaff}, $P_v$ contains at least~$\lvert P_v \rvert - \varepsilon k \geq 2k - r - \varepsilon k$ $k$-unaffiliated vertices.  Since, by assumption, $v$ and $w$ have at most~$2k - r - 2$ $k$-unaffiliated vertices, by Lemma~\ref{GeodesicDisj}, we must have
\[
\lvert P_v \rvert \leq 2k - r + \varepsilon k.
\]
Observe that $G$ contains at most~$\varepsilon k$ vertices that are $k$-unaffiliated with $v$~and~$w$ and are not on $P_v$, or else we are done by Corollary~\ref{PvManyUnaff}. The same assertions as above hold for $P_w$ in place of~$P_v$.  By Lemma~\ref{GeodesicDisj}, for the vertex sets of the paths,
\begin{equation}\label{eq:GeodesicsBd}
\lvert P \cap P_v \rvert \geq (1 - 2\varepsilon)k \text{ and } \lvert P \cap P_w \rvert \geq (1 - 2\varepsilon)k.
\end{equation}

Let $u$ be the point furthest from $v$ at which $P$ and $P_v$ coincide.  By hypothesis, $u \in P_1$.  Similarly, let $u'$ be the furthest point from $v$ at which $P'$ and $P_v$ coincide.  By Lemma~\ref{GeodesicDisj}, $d(u,w)\leq 2 \varepsilon k$ and $d(u',w')\leq 2 \varepsilon k$.  It follows that
\begin{equation}\label{eq:disttovbds}
(1 - 2\varepsilon)k \leq d(v,u), d(v,u') \leq k.
\end{equation}
Now we consider two cases: when $u' \in P_1$ and when $u' \in P_2$.

Suppose that $u'$ lies on $P_1$.  Suppose first that $u$ is closer to~$v$ than $u'$ is (see Figure~\ref{fig:distalt1}).  Then $d(u,w') = d(u,w) \leq 2\varepsilon k$, thus,
\[
d(w,w') \leq d(w,u) + d(u,w') \leq 4\varepsilon k.
\]
If $u'$ is closer to~$v$ than $u$ is, then by a similar argument, $d(u',w) = d(u',w') \leq 2\varepsilon k$, and so
\[
d(w,w') \leq d(w,u') + d(u',w') \leq 4\varepsilon k.
\]

\begin{figure}
\centering
\includegraphics{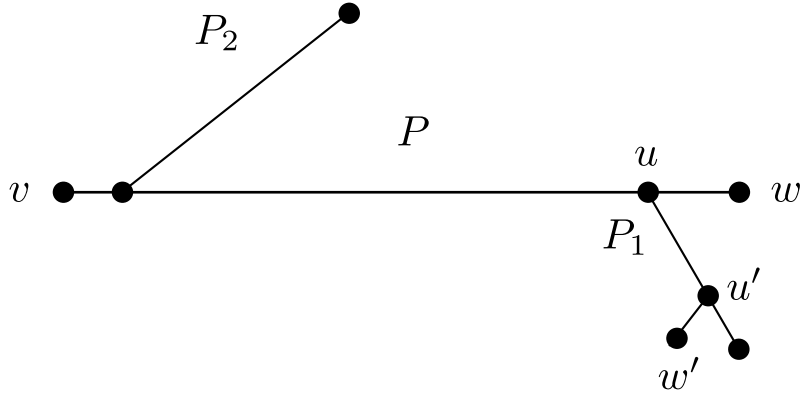}
\caption[The value of~$d(w,w')$ when $u' \in P_1$]{When $u' \in P_1$, all $k$-neighbours of~$v$ are close together.  The dashed segments represent the path~$P_1$.}\label{fig:distalt1}
\end{figure}

If, however, $u' \in P_2$, then $d(w, w')$ depends on the length of~$P_w \setminus P$.  Because $u' \in P_2$, we have $\lvert P_2 \rvert \geq (1-2\varepsilon)k$.  Since $\lvert P_1 \rvert + \lvert P_2 \rvert = \lvert P_v \rvert + 1 \leq 2k + \varepsilon k$ (and similarly for $\abs{P_3},~\abs{P_4}$ and~$\abs{P_w}$), we have
\begin{equation}\label{eq:Pilength}
(1 - 2\varepsilon)k \leq \lvert P_i \rvert \leq (1 + 3\varepsilon)k \text{ for } i=1,2.
\end{equation}
Now we consider the geodesics $P_3$ and~$P_4$ that comprise $P_w$.  We have assumed that $P_3$ contains $P_w\cap P$. Then, because $P_4 \cap P = \emptyset$, we have $P_1 \cap P_4 \subseteq P_1 \setminus P$, hence,
\[
\lvert P_1\cap P_4 \rvert \leq \lvert P_1\setminus P \rvert \leq (1 + 3\varepsilon)k - (1 - 2\varepsilon)k = 5\varepsilon k.
\] 
Thus, by Lemma~\ref{GeodesicDisj},
\[
\lvert P_4 \setminus P_2 \rvert = \lvert P_4\cap P_1 \rvert + \lvert P_4 \setminus P_v \rvert \leq 5\varepsilon k + 2\varepsilon k = 7 \varepsilon k.
\] 

Now we shall show that if $\abs{P_4}$ is at all large, then $d(w, w') = o(k)$, while if $\abs{P_4}$ is very small, then $d(w, w') = 2k - o(k)$.

\emph{Case 1:} Suppose first that $\lvert P_4 \rvert > 7 \varepsilon k$. Then, because $\abs{P_4 \setminus P_2} \leq 7\varepsilon k$, we have $P_2 \cap P_4 \neq \emptyset$.  Let $t$ denote the vertex at which $P_4$ meets $P$.  Then by Lemma~\ref{GeodesicDisj}, we have $d(t, w) < 2\varepsilon k$.  Let $q$ be the vertex of~$P_2\cap P_4$ that is closest to~$w$ (see Figure~\ref{fig:distalt2}).  Then $d(q,w) \leq \abs{P_4 \setminus P_2} + d(t, w) < 9 \varepsilon k$.  Then
\[
k + 9\varepsilon k \geq d(q,v)\geq k - 9\varepsilon k.
\]

\begin{figure}
\centering
\includegraphics{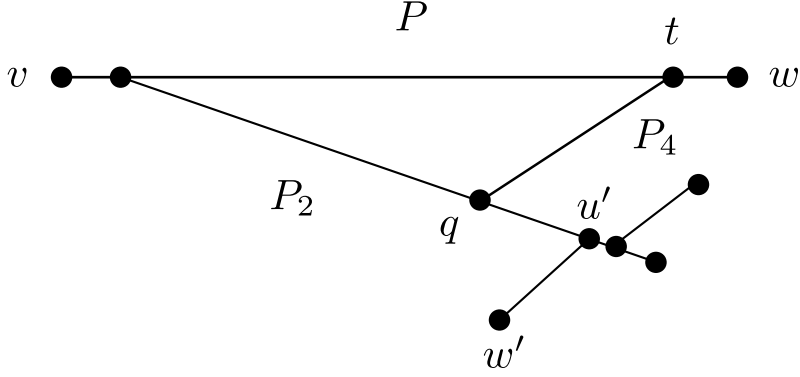}
\caption[The value of~$d(w,w')$ when $u' \in P_2$, first case]{When $u' \in P_2$ and $\abs{P_4}$ is large, all $k$-neighbours of~$v$ are again close together.  The dashed segments represent the path~$P_4$.  The vertices $a$ and~$b$ denote the endpoints of~$P_2$.}\label{fig:distalt2}
\end{figure}

Now we bound $d(q, u')$.  The bound depends on the location of~$q$ relative to~$u'$.  If $q$ is farther away from $v$ than $u'$ is, then by~\eqref{eq:Pilength},
\[
d(q,v) \leq \abs{P_2} - d(q,u') \leq (1 + 3\varepsilon)k - (1 - 2\varepsilon)k = 5\varepsilon k.
\]
If, however, $q$ is closer to~$v$ than $u'$ is, then the fact that $d(q, v) \geq k - 9\varepsilon k$ implies that $d(q, u') < d(q, w') \leq 9\varepsilon k$.  Thus, $d(q,u')\leq 9 \varepsilon k$ and $d(q,w')\leq 11 \varepsilon k$. We therefore have 
\[
d(w,w') \leq d(w,q) + d(q,w') \leq 9\varepsilon k + 11\varepsilon k = 20\varepsilon k,
\]
which completes the proof of Case 1.

\emph{Case 2:} Suppose instead that $\lvert P_4 \rvert \leq 7 \varepsilon k$. Then, by~\eqref{eq:GeodesicsBd} and Lemma~\ref{GeodesicDisj}, we have
\begin{equation}\label{eq:P2P3bd}
\lvert P_2 \cap P_3 \rvert = \lvert P_2 \rvert - \lvert P_2 \setminus P_w \rvert - \lvert P_2 \cap P_4 \rvert \geq (k - 2\varepsilon k) - 2\varepsilon k - 7\varepsilon k = k - 11\varepsilon k.
\end{equation}
Since $P_3$ contains $P_w \cap P$ and $P_2$ is edge-disjoint from~$P$, it follows from~\eqref{eq:GeodesicsBd} and~\eqref{eq:P2P3bd} that
\begin{equation}\label{eq:P3partsbd}
\lvert P_2 \cap P_3 \rvert + \lvert P \cap P_3 \rvert \geq 2k - 13\varepsilon k.
\end{equation}
Let $q'$ be the vertex of~$P_2 \cap P_3$ that is furthest away from $w$ (see Figure~\ref{fig:distalt3}).  Then, by~\eqref{eq:Pilength} and~\eqref{eq:P3partsbd},
\begin{equation}\label{eq:wtoP2bd}
2k + 3\varepsilon k \geq d(q', v) + d(v, w) \geq d(q',w) \geq \lvert P_2 \cap P_3 \rvert + \lvert P \cap P_3 \rvert \geq 2k - 13\varepsilon k.
\end{equation}
Also, it follows from~\eqref{eq:Pilength} and~\eqref{eq:P2P3bd} that
\[
k + 3 \varepsilon k \geq \abs{P_2} \geq d(q',v) \geq \abs{P_2 \cap P_3} \geq k - 11 \varepsilon k.
\]
It follows from this and~\eqref{eq:disttovbds} that $d(q',u')\leq 13\varepsilon k$ and therefore that $d(q',w') \leq 15\varepsilon k$. We obtain from this and~\eqref{eq:wtoP2bd} that
\[
d(w,w') \geq d(q',w) - d(q',w') \geq 2k - 28 \varepsilon k.
\]
This proves the lemma.
\end{proof}

\begin{figure}
\centering
\includegraphics{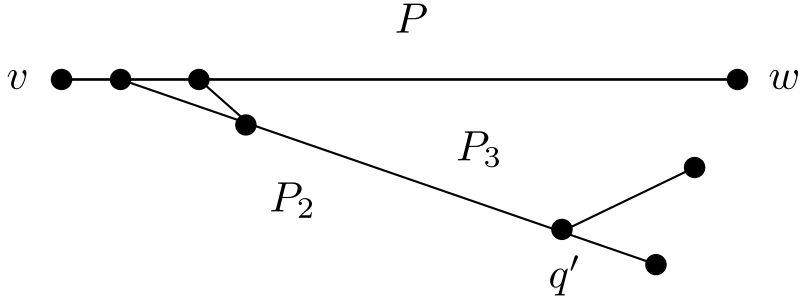}
\caption[The value of~$d(w,w')$ when $u' \in P_2$, second case]{When $u' \in P_2$ and $\abs{P_4}$ is small, all $k$-neighbours of~$v$ are at distance~$(2 - o(1))k$ from one another.  The dashed segments represent the portion of the path~$P_3$ that is disjoint from~$P$.  The vertices $a$ and~$b$ denote the endpoints of~$P_2$.}\label{fig:distalt3}
\end{figure}

We shall now show that the assumption that $d(w,w') = (2 - o(1))k$ for some $w'$ leads to a contradiction.  Let $\delta = O(\varepsilon) = o(1)$.  For $v \in V(G)$, define the {\em cluster} of~$v$ to be the set~$\cluster_v$ of vertices at distance at most~$\delta k$ from $v$.

\begin{lemma}\label{clusters}
Fix $\varepsilon > 0$.  For $k$ and $n$~large enough, let $G$ be a graph on $n$ vertices and let $v$ and $w$ be $k$-neighbours with fewer than~$2k - r - 2$ $k$-unaffiliated vertices.  Then every $k$-neighbour of~$v$ is at distance~$O(\varepsilon k) = o(k)$ from $w$, and vice versa.
\end{lemma}

\begin{proof}
Suppose that $w' \neq w$ is a $k$-neighbour of~$v$ such that $d(w,w') = (2 - o(1))k$.  In this case, our graph~$G$ is `flat', i.e., it contains a geodesic~$P_3$ of length~$(2 - o(1))k$ and every vertex is $o(k)$ away from $P_3$. The vertex~$v$ lies close to the centre of~$P_3$, whereas $w$ and $w'$ lie near the opposite ends: every $k$-neighbour of~$v$ lies within~$o(k)$ of either $w$ or~$w'$. Since there are no vertices at distance~$(2 - o(1))k$ from $v$, switching $v$ and $w$ in the statement of Lemma~\ref{distalt} yields that all $k$-neighbours of~$w$ are close to~$v$. Since, by assumption, $v$ and $w$ have fewer than~$2k$ $k$-unaffiliated vertices, all but at most~$2k$ vertices are contained in one of the clusters $\cluster_v$,~$\cluster_w$ and~$\cluster_{w'}$. Every vertex on $P_3$ not lying within~$2\delta k$ of either $v$,~$w$ or~$w'$ cannot have a $k$-neighbour in any of these clusters. Therefore any such vertex has at most~$2k$ $k$-neighbours. So, we have $c = (2 - o(1))k$ vertices of $k$-degree at most~$2k$. Applying Tur\'an's theorem to the remaining vertices, we obtain
\[e(G_k)\leq \left(\dfrac{n - c}{2}\right)^2 + 2ck < \left(\dfrac{n - k + 1}{2}\right)^2,\]
provided that $n$ is sufficiently large compared to~$k$, contradicting our hypothesis that $e(G_k) \geq (n - k + 1)^2 / 4$.
\end{proof}

Since all but at most~$2k$ of the vertices in $G$ are at distance~$k$ from either $v$ or~$w$, we can conclude that all but at most~$2k$ vertices lie either in $\cluster_v$, that is, within distance~$o(k)$ of~$v$, or in $\cluster_w$. This is a fairly strong structural property of~$G$ and from here it is a short step to completing the proof of Theorem~\ref{thm:Dbroom}.

\begin{proof}[Proof of Theorem~\ref{thm:Dbroom}.]
Observe that by the definition of~$\delta$, every vertex on the $vw$-path~$P$ that is at distance more than~$\delta k$ from both $v$ and $w$ cannot have a $k$-neighbour in either cluster.  We shall use these vertices, together with a small set that we shall now construct, to produce a contradiction as in the proof of Lemma~\ref{clusters}.

We shall now show that $\abs{P_2}$ must be very small and that $\abs{P_1}$ cannot be much larger than~$k$.  It will then follow that $\abs{P_v}$ cannot be much larger than~$k$, either.

If $\lvert P_2 \rvert > 10 \delta k$, then $P_2$ contains $4\delta k$ vertices at distance between $2\delta k$~and~$6\delta k$ from $v$.  By Lemma~\ref{GeodesicDisj}, at least~$4\delta k - 2\varepsilon k > 3\delta k$ of these vertices are contained either in $P_3$ or in $P_4$.  We shall consider two cases: when at least~$\delta k$ of these vertices are in $P_3$ and when at least~$2 \delta k$ of them are in $P_4$.  We shall show below that either case produces a contradiction as in the proof of Lemma~\ref{clusters}.  Suppose first that $\delta k$ of them are in $P_3$.  Let $S_1 = \{x \in P_3 : 2\delta k \leq d(v, x) \leq 6\delta k\}$, let $y$ be the closest vertex of~$P_3$ to~$v$ and let $s \in P_2$ be such that $d(v, s) = 2\delta k$ (see Figure~\ref{fig:outside1}).  By Lemma~\ref{GeodesicDisj} and the triangle inequality, we have $d(y, s) \geq 2\delta k - 2\varepsilon k$.  So,
\[
d(w, s) \geq k - 2\varepsilon k + 2\delta k - 2\varepsilon k \geq k + \delta k,
\]
which means that the vertices of~$S_1$ are all at distance at least~$k + \delta k$ from $w$ and therefore have no $k$-neighbours in the clusters.

\begin{figure}
\centering
\includegraphics{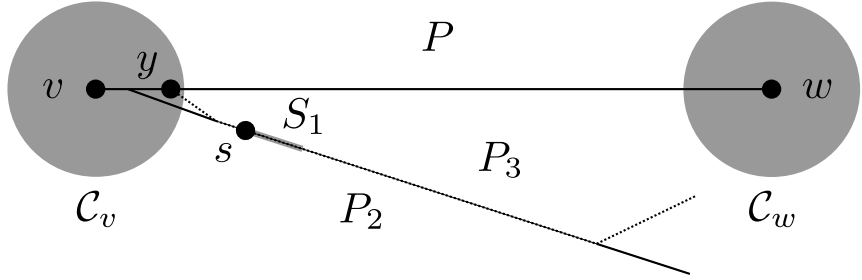}
\caption[Vertices of~$P_v$ just outside of the cluster of~$v$]{The vertices of~$S_1$, shown in grey, are just outside of $\cluster_v$, and so have no $k$-neighbours in either $\cluster_v$ or~$\cluster_w$.}\label{fig:outside1}
\end{figure}

Suppose instead that $2 \delta k$ of them are in $P_4$. We shall show that this implies that $\abs{P_4} \sim k$.  Let $x$ be the closest vertex of~$P_4$ to~$w$ and let $z$ be the point at which $P_4$ meets $P_2$. Let $a$ be the closest point of~$P_2$ to~$v$ (see Figure~\ref{fig:outside2}). We may assume that
\[3\delta k < 4\delta k - 2\varepsilon k \leq d(a,z) \leq d(v,z) \leq 4\delta k.\]
Then, because $a$, $x \in P$, 
\begin{equation}\label{eq:P_4LB}
\lvert P_4 \rvert \geq d(x,z) \geq d(a,x) - d(a,z) \geq k - 4\varepsilon k - 4\delta k.
\end{equation}
Also, since $\lvert P_w \rvert \leq (2 + \varepsilon)k$ and $\lvert P_3 \rvert \geq (1 - 2\varepsilon)k$, it follows that
\begin{equation}\label{eq:P_4UB}
\lvert P_4 \rvert \leq (1 + 3\varepsilon)k.
\end{equation}
Then $\lvert P_4 \rvert \sim k$.  Let $S_2$ denote the $2\delta k$ central vertices of~$P_4$.  Observe that if $q \in S_2$, then both $d(q,w)$ and $d(q,v)$ are bounded away from both~$0$ and~$k$. Indeed, after a bit of calculation, it follows from~\eqref{eq:P_4LB} and~\eqref{eq:P_4UB} that
\[
\dfrac{k - 4\varepsilon k - 4\delta k}{2} - \delta k \leq d(q,w) \leq \dfrac{(1 + 3\varepsilon)k}{2} + \delta k
\]
and that
\[
3\delta k + \dfrac{k - 4\varepsilon k - 4\delta k}{2} - \delta k \leq d(q,v) \leq 4\delta k + \dfrac{(1 + 3\varepsilon)k}{2} + \delta k.
\]

\begin{figure}
\centering
\includegraphics{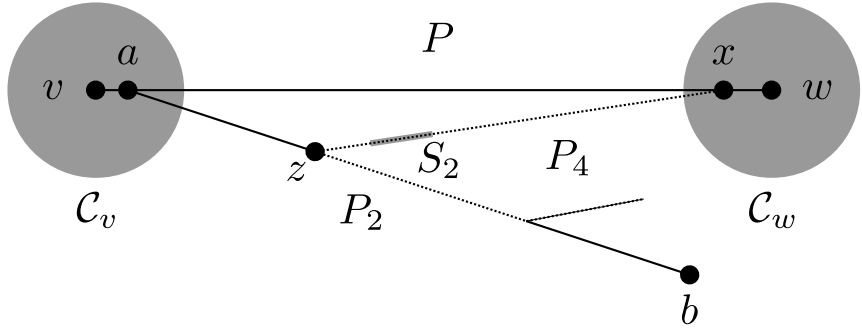}
\caption[The central vertices of~$P_4$ are outside of both clusters]{The vertices of~$S_2$, shown in grey, have no $k$-neighbours in either cluster.}\label{fig:outside2}
\end{figure}

Together with the aforementioned vertices on $P$, in each case we obtain $c \geq k + \delta k$ vertices that are not in either cluster and whose $k$-degree is at most~$2k$. By applying Tur\'an's theorem to these vertices as in the proof of Lemma~\ref{clusters}, we obtain a contradiction for large enough~$n$. Thus, $\lvert P_2 \rvert \leq 10 \delta k$. 

Similarly, we claim that $\lvert P_1 \rvert \leq k + 10 \delta k$. If not, then $P_1$ contains at least~$3\delta k$ vertices that are at distance between $k+2\delta k$~and~$k+5\delta k$ from $v$, and at distance between $2\delta k - 2\varepsilon k$~and~$5\delta k + 2\varepsilon k$ from $w$, and we obtain a contradiction as above.

We therefore have 
\[\lvert P_v \rvert \leq k + 20\delta k.\]
Recall that by Lemma~\ref{SpTreePathLen} we have $\abs{P_v} \geq 2k - r$.  If $2k - r > k + 20\delta k$, then we have a contradiction, which means that $r \geq k - 1$, that is, that $G$ has at least~$k - 1$ interior vertices.  Hence, by~\eqref{eq:CSint}, we have $e(G_k) = (n - k + 1)^2 / 4$.  Let $I$ denote the set of interior vertices of~$G$, and recall that by definition an interior vertex is isolated in $G_k$.  Because $e(G_k)$ is maximal, it follows from~\eqref{eq:CSintunaff} that for all~$\{x, y\} \in E(G_k)$, the only vertices that are $k$-unaffiliated with both $x$~and~$y$ are those on the (unique) shortest path between $x$ and $y$.  Hence, all of these vertices must be in $I$.  It follows that $G_k \setminus I$ is a complete bipartite graph with balanced parts, which means that $G$ is $k$-isomorphic to the double broom, as claimed.

Otherwise, for $k$ and $n$~large enough, if any geodesic outside of~$P$ has length more than~$20\delta k + 4\epsilon k \leq 21\delta k$, then it has at least~$2\varepsilon k$ vertices disjoint from~$P_v$, and we are done by Lemma~\ref{GeodesicDisj}.

Let $m$ be the midpoint of~$P$, or one of the midpoints if $k$ is odd. Suppose that there exist $k$-neighbours $x$ and~$y$ in $G$ such that every $k$-geodesic between $x$ and $y$ misses $m$. Let $Q$ be such a geodesic. Then it must miss either the path~$P_{vm}$ between $v$ and $m$ or the path~$P_{wm}$ between $w$ and $m$. In either case, $Q$ will contain a geodesic of length at least~$k/4$ disjoint from~$P$, a contradiction. Hence, every pair of $k$-neighbours is connected by a geodesic containing~$m$.

It follows that $G$ is $k$-isomorphic to~$T_m$, the breadth-first tree with respect to~$m$. By Theorem~\ref{thm:treemax}, $T_m$ has at most as many $k$-distances as the double broom does. Thus, $e(G) = (n - k + 1)^2 / 4$ and $G$ is $k$-isomorphic to the double broom. The proof of Theorem~\ref{thm:Dbroom} is complete.
\end{proof}

\section{Discussion}\label{disc}

We believe that Conjecture~\ref{Kt+1Free} may be susceptible to a similar approach to the one above. In Section~\ref{prelim}, we obtained our first non-trivial bound, Lemma~\ref{CS}, by adapting a proof of Mantel's theorem, which is the simplest case of Tur\'an's theorem. Unfortunately we were not able find a straightforward generalisation of this approach to $K_{t+1}$-free distance-$k$ graphs when $t \geq 3$. A possible solution might be to adapt a proof of Tur\'an's theorem that works for all~$t$.  However, it seems difficult to generalise Lemma~\ref{SpTreePathLen} for values of $t$~greater than 2.

We also note that in~\cite{csik}, Csikv\'ari asked a similar question about maximising or minimising the number of (closed) \emph{walks} of length~$k$ in a connected graph~$G$ on $n$ vertices and $m$ edges. In the same paper Csikv\'ari settled the case of closed walks and $m = n-1$, that is, when $G$ is a tree. The answer for general walks on trees was given in~\cite{bb-mt}, but the general case remains open.

\section{Acknowledgments}\label{ack}
We would like to thank B\'ela Bollob\'as for introducing us to the problem and for helpful discussions.  We would like to thank the anonymous referees for their helpful comments.  We would also like to thank Paul Balister for helpful comments and for running the computer search mentioned after Example~\ref{ex:G2max}.

\bibliographystyle{plain}
\bibliography{DistancesBib,GraphBib}

%
%
%
%
%
%
%
%
%
%
%
%

\end{document}